\documentclass[11pt,oneside,english]{amsart}
\usepackage[T1]{fontenc}
\usepackage{amsthm}
\usepackage{amstext}
\usepackage{amssymb}
\usepackage{esint}

\makeatletter
\numberwithin{equation}{section}
\numberwithin{figure}{section}
\theoremstyle{plain}
\newtheorem{thm}{\protect\theoremname}
  \theoremstyle{plain}
  \newtheorem{prop}[thm]{\protect\propositionname}
  \theoremstyle{remark}
  \newtheorem{rem}[thm]{\protect\remarkname}
  \theoremstyle{plain}
  \newtheorem{lem}[thm]{\protect\lemmaname}
  \theoremstyle{plain}
  \newtheorem*{thm*}{\protect\theoremname}
  \theoremstyle{plain}
  \newtheorem{cor}[thm]{\protect\corollaryname}

\makeatother

\usepackage{babel}
  \providecommand{\corollaryname}{Corollary}
  \providecommand{\lemmaname}{Lemma}
  \providecommand{\propositionname}{Proposition}
  \providecommand{\remarkname}{Remark}
  \providecommand{\theoremname}{Theorem}
\providecommand{\theoremname}{Theorem}

\begin{document}
\global\long\def\bbZ{\mathbb{Z}}

\global\long\def\bbN{\mathbb{N}}

\global\long\def\bbX{\mathbb{X}}

\global\long\def\cF{\mathcal{F}}

\global\long\def\cB{\mathcal{B}}

\global\long\def\fT{\mathfrak{\tau}}

\global\long\def\cCT{\mathcal{C}_{\fT-1}}

\global\long\def\cC{\mathcal{C}}

\global\long\def\RR{\mathbb{R}}

\global\long\def\BB{\mathcal{B}}

\global\long\def\BB{\mathcal{B}}

\global\long\def\RR{\mathbb{R}}

\global\long\def\Tau{\mathcal{T}}

\global\long\def\Ra{\mathcal{R}}

\global\long\def\NN{\mathbb{N}}

\global\long\def\ZZ{\mathbb{Z}}

\global\long\def\P{\mathrm{P}}

\global\long\def\C#1{\mathcal{#1}}

\global\long\def\ZZ{\mathbb{Z}}

\global\long\def\RN#1{\left(T^{#1}\right)'}

\global\long\def\F#1{\mathcal{F}_{#1}}

\global\long\def\lloc{\ll^{loc}}

\global\long\def\NN{\mathbb{N}}

\global\long\def\PP{\mathbb{P}}

\global\long\def\lsup{\varlimsup_{n\to\infty}}

\global\long\def\ubeta{\overline{\beta}}

\global\long\def\lbeta{\underline{\beta}}

\global\long\def\linf{\varliminf_{n\to\infty}}

\global\long\def\lsupK{\varlimsup_{n\to\infty,n\in K}}

\global\long\def\linfK{\varliminf_{n\to\infty,n\in K}}

\title[gap for symmetric Birkhoff sums]{A universal divergence rate for symmetric Birkhoff Sums in infinite
ergodic theory}

\author{Zemer Kosloff}

\curraddr{Mathematics institute, University of Warwick, Coventry, CV47AL, United
Kingdom. }

\email{z.kosloff@warwick.ac.uk }
\begin{abstract}
We show that there exists a universal gap in the failure of the ergodic
theorem for symmetric Birkhoff sums in infinite ergodic theory. 
\end{abstract}

\maketitle

\section{Introduction}

For an ergodic infinite measure preserving system, the ergodic theorem
fails in the sense that there does not exist a normalizing sequence
for its Birkhoff sums \cite{Aaronson no et}. That is for every conservative,
ergodic, measure preserving system $\left(X,\BB,m,T\right)$ with
$m(X)=\infty$ , $0\leq f\in L_{1}\left(X,m\right)$ and $a_{n}\to\infty$,
either 
\[
\liminf_{n\to\infty}\frac{S_{n}(f)}{a_{n}}=0\ a.e.
\]
or
\[
\limsup_{n\to\infty}\frac{S_{n}(f)}{a_{n}}=\infty\ a.e.
\]
Here $S_{n}(f):=\sum_{k=0}^{n-1}f\circ T^{k}$ denotes the Birkhoff
sum of $f$. For an invertible transformation one can consider symmetric
(two-sided) Birkhoff sums 
\[
\Sigma_{n}\left(f\right)(x):=\sum_{|k|<n}f\left(T^{k}x\right),
\]
where the summation is in a symmetric time interval. The papers \cite{Aaronson-Kosloff-Weiss,F. Maucourant and B. Schapira}
contain examples of infinite measure preserving transformations for
which there exists normalizing constants $a_{n}\to\infty$ such that
for every $0\leq f\in L_{1}(X,m)$, 
\begin{equation}
\linf\frac{\Sigma_{n}(f)}{a_{n}}>0\ a.e.\ \text{and }\lsup\frac{\Sigma_{n}(f)}{a_{n}}<\infty\ a.e.\label{eq: Symmetric boundedness}
\end{equation}
The examples of \cite{Aaronson-Kosloff-Weiss} include some natural
transformations in infinite ergodic theory such as the class of rank
one transformations with bounded cutting sequence and generalized
recurrent events with a certain trimmed sum property (some null recurrent
Markov chains are in this class). This shows that symmetric Birkhoff
sums can behave better than their one sided counterparts. However
in the work with Jon Aaronson and Benjamin Weiss we proved that for
an invertible infinite measure preserving transformation, there is
no ergodic theorem for symmetric Birkhoff sums. That is for every
normalizing sequence $a_{n}\to\infty$ and $0\leq f\in L_{1}(X,m)$
if
\\
$0<\linf\frac{1}{a_{n}}\Sigma_{n}(f)(x)<\infty$,
then 
\[
\linf\frac{\Sigma_{n}(f)}{a_{n}}<\lsup\frac{\Sigma_{n}(f)}{a_{n}}\ a.e.
\]
The purpose of this note (which is largely taken from the authors
Ph.D. thesis) is to prove a universal quantitative divergence rate
for symmetric Birkhoff sums. 
\begin{thm}
\label{thm: Main Theorem}For every conservative, ergodic, measure
preserving system $\left(X,\BB,m,T\right)$ with $m(X)=\infty$ ,
$0\leq f\in L_{1}\left(X,m\right)$ and $a_{n}\to\infty$, if
\\
 $0<\linf\frac{1}{a_n}\Sigma_{n}(f)(x)<\infty$, then 
\[
\frac{\linf\frac{\Sigma_{n}(f)}{a_{n}}}{\lsup\frac{\Sigma_{n}(f)}{a_{n}}}\leq\frac{10001}{10002}.
\]
 
\end{thm}
After proving the theorem we give an application to the study of fluctuations
of symmetric Birkhoff integrals of horocyclic flows on geometrically
finite surfaces.

\subsubsection*{Notation}

From now on we will write 
\[
S_{n}^{-}(f):=\sum_{k=1}^{n-1}f\circ T^{-k}=\Sigma_{n}(f)-S_{n}(f).
\]
For eventually positive sequences $a_{n},b_{n}$ we write:
\begin{itemize}
\item $a_{n}\sim b_{n}$ if $\lim_{n\to\infty}\frac{a_{n}}{b_{n}}=1$. 
\item $a_{n}\lesssim b_{n}$ if $\lsup\frac{a_{n}}{b_{n}}\leq1$. 
\item $a_{n}\asymp b_{n}$ if there exists $C>1$ such that $C^{-1}b_{n}\leq a_{n}\leq Cb_{n}$
for all $n\in\NN$. 
\item For an infinite subset $K\subset\NN$, $a_{n}\underset{n\in K}{\lesssim}b_{n}$
if $\lsupK\frac{a_{n}}{b_{n}}\leq1.$
\item $a=b\pm\epsilon$ means $b-\epsilon<a<b+\epsilon$. 
\item Given a standard $\sigma$-finite measure space $\left(X,\BB,m\right)$
and a subcollection of sets $\mathcal{C\subset B}$, we write $\mathcal{C}_{+}$
to be the collection of sets $A\in\mathcal{C}$ of positive measure.
\item $L_{1}(X,m)_{+}$ is the collection of nonnegative integrable functions. 
\item All transformations or flows in this paper are assumed to be invertible.
\end{itemize}

\section{Preliminaries }

\subsubsection*{Bounded Rational Ergodicity }

As in \cite{Aar}, a conservative, ergodic, measure preserving transformation
$\left(X,\BB,m,T\right)$ is called \textit{boundedly rationally ergodic}
(\texttt{BRE}) if $\exists A\in\BB,$ $0<m(A)<\infty$ and $M<\infty$
so that 
\begin{eqnarray}
S_{n}\left(1_{A}\right)(x) & \leq & Ma_{n}(A)\ \text{a.e.}\ on\ A\ \forall n\geq1\label{eq:BRE}\\
 &  & \text{where}\ a_{n}(A)=\sum_{k=0}^{n-1}\frac{m\left(A\cap T^{-k}A\right)}{m(A)^{2}}.\nonumber 
\end{eqnarray}
In this case \cite{Aar}, $\left(X,\BB,m,T\right)$ is \textit{weakly
rationally ergodic} (\texttt{WRE}), that is, writing $a_{n}(T):=a_{n}(A)$
(where $A$ as in (\ref{eq:BRE})), there is a dense hereditary ring
\[
\mathcal{R}(T)\subset\cF:=\left\{ F\in\BB:\ m(F)<\infty\right\} 
\]
(including all sets satisfying (\ref{eq:BRE})) so that 
\[
a_{n}(F)\sim a_{n}\left(T\right)\ \forall F\in\mathcal{R}(T),\ m(F)>0
\]
and 
\[
\sum_{k=0}^{n-1}m\left(F\cap T^{-k}G\right)\sim m(F)m(G)a_{n}(T),\ \ \forall F,G\in\mathcal{R}(T).
\]
For invertible transformations, one can define similarly the two sided
analougs of the properties \texttt{BRE} and \texttt{WRE}. $\left(X,\BB,m,T\right)$
is:
\begin{itemize}
\item two sided, boundedly rationally ergodic if $\exists A\in\BB$, $0<m(A)<\infty$ and $M<\infty$
so that 
\begin{eqnarray}
\Sigma_{n}\left(1_{A}\right)(x) & \leq & M\overline{a}_{n}(A)\ \text{a.e.}\ on\ A\ \forall n\geq1\label{eq: two sided BRE}\\
 &  & \text{where}\ \overline{a}_{n}(A)=\sum_{|k|\leq n}\frac{m\left(A\cap T^{k}A\right)}{m(A)^{2}}\sim2a_{n}(A).\nonumber 
\end{eqnarray}

\item two sided, weakly rationally ergodic if there is a dense hereditary
ring 
\[
\overline{\mathcal{R}}(T)\subset\cF
\]
(including all sets satisfying (\ref{eq: two sided BRE})) so that
\[
\overline{a}_{n}(F)\sim2a_{n}\left(T\right)\ \forall F\in\overline{\mathcal{R}}(T),\ m(F)>0.
\]

\end{itemize}
If $\left(X,\BB,m,T\right)$ is one sided \texttt{BRE} then so is
$\left(X,\BB,m,T^{-1}\right)$. This can be seen for example by the
fact that if $A\in\BB$ is the set along which (\ref{eq:BRE}) holds,
then for $n\in\mathbb{N}$ and $x\in A$ one can define 
\[
k(x,n):=\max\left\{ k\in\NN\cup\{0\}:\ k< n\ \text{and}\ T^{-k}x\in A\right\} .
\]
It then follows that for all $n\in\mathbb{N}$ and $x\in A$, 
\begin{eqnarray*}
\sum_{k=-n+1}^{0}1_{A}\circ T^{k}(x) & = & S_{k(x,n)}\left(1_{A}\right)\circ T^{-k(x,n)}(x)\\
 & \overset{\eqref{eq:BRE}}{\leq} & Ma_{k(x,n)}(A)\ \ \text{since}\ T^{-k(x,n)}x\in A\\
 & \leq & Ma_{n}(A),\ \ \ \ \text{since}\ k(x,n)\leq n.
\end{eqnarray*}
Therefore in the case of invertible transformations:
\begin{itemize}
\item $\left(X,\BB,m,T\right)$ is two sided \texttt{BRE} if and only if
it is (one sided) bounded rationally ergodic.
\item If $\left(X,\BB,m,T\right)$ is two sided \texttt{BRE,} then it is
two sided \texttt{WRE} and for $F,G\in\bar{\mathcal{R}}\left(T\right)$,
\begin{eqnarray}
\int_{F}\frac{\Sigma_{n}\left(1_{G}\right)}{2a_{n}(T)}dm & = & \frac{1}{2a_{n}\left(T\right)}\sum_{k=-n}^{n}m\left(F\cap T^{-k}G\right)\label{eq: weak rational ergodicity for symmetric sums}\\
 & \sim & m(F)m(G)\ \text{as}\ n\to\infty.\nonumber 
\end{eqnarray}

\end{itemize}

\subsubsection{Some observations:}

Let $\left(X,\BB,m,T\right)$ be a conservative, ergodic measure preserving
transformation.
\begin{enumerate}
\item By the ratio ergodic theorem, for all $f,g\in L_{1}(X,m)$ with $g>0$,
\[
\frac{S_{n}\left(f\right)}{S_{n}(g)}(x)\xrightarrow[n\to\infty]{}\frac{\int_{X}fdm}{\int_{X}gdm},\ \text{for a.e.}\ x
\]
and by a similiar argument for $T^{-1}$, 
\[
\frac{\Sigma_{n}\left(f\right)}{\Sigma_{n}(g)}(x)\xrightarrow[n\to\infty]{}\frac{\int_{X}fdm}{\int_{X}gdm},\ \text{for a.e.}\ x.
\]
A consequence of this is that in order to check if (\ref{eq: Symmetric boundedness})
holds for a sequence $a_{n}\to\infty$, it is enough to check if it
holds for one function $f\in L_{1}\left(X,m\right)_{+}$. Variants
of this application of the ratio ergodic theorem appear throughout
this work.
\item For $a_{n}\to\infty$ and $f\in L_{1}(X,m)_{+},$ the functions $\lsup\frac{\Sigma_{n}(f)}{a_{n}}$
and $\linf\frac{\Sigma_{n}(f)}{a_{n}}$ are $T$ invariant, hence
constant almost everywhere. 
\item As in the one sided case $\left(X,\BB,m,T\right)$ is two sided \texttt{BRE
}if and only if for all $f\in L_{1}(X,m)_{+}$,
\[
\lsup\frac{1}{2a_{n}(T)}\Sigma_{n}(f)<\infty\ a.e.
\]

\end{enumerate}
In case $T$ is bounded rationally ergodic, there exists $\underline{\beta}=\underline{\beta}(T)\in[0,1],$
$\alpha=\alpha(T)$ and $\bar{\beta}=\overline{\beta}(T)\in[1,\infty)$
so that $\forall f\in L_{1}(X,m)_{+}$ for $m$ a.e. $x$: 
\begin{eqnarray*}
\lsup\frac{1}{a_{n}(T)}S_{n}(f)(x) & = & \alpha\int_{X}fdm\\
\lsup\frac{1}{2a_{n}(T)}\Sigma_{n}(f)(x) & = & \ubeta\int_{X}fdm\\
\linf\frac{1}{2a_{n}(T)}\Sigma_{n}(f)(x) & = & \lbeta\int_{X}fdm.
\end{eqnarray*}

We will make use of the following proposition from \cite{Aaronson-Kosloff-Weiss}. 
\begin{prop}\label{prop: AKW}
\cite[Prop. 1]{Aaronson-Kosloff-Weiss} Let $(X,\BB,m,T)$ be an invertible,
conservative, ergodic, measure preserving transformation. \\
(i) If $T$ satisfies (\ref{eq: Symmetric boundedness}) w.r.t. to
some normalizing constants $a_{n}\to\infty$, then $T$ is bounded
rationally ergodic and $a_{n}\asymp2a_{n}(T)$. \\
(ii) If $T$ is bounded rationally ergodic, then 
\[
\alpha(T)=\alpha\left(T^{-1}\right),
\]
whence 
\begin{eqnarray}
\ubeta(T) & \leq & \alpha\left(T\right)\leq2\ubeta\left(T\right)\ \text{and}\label{eq: beta gag is larger than alpha}\\
\lbeta(T) & \leq & \frac{\alpha(T)}{2}.\label{eq: beta rizpa is smaller than alpha/2}
\end{eqnarray}
\end{prop}

A consequence of this proposition
is that using the convention that $\frac{a}{\infty}=0$ for all $0\leq a<\infty$,
if $T$ is not bounded rationally ergodic then for any $a_{n}\to\infty$
and $f\in L_{1}\left(X,m\right)_{+}$ either $\linf\frac{1}{a_{n}}\Sigma_{n}(f)(x)=0$
or $\lsup\frac{1}{2a_{n}(T)}\Sigma_{n}(f)(x)=\infty$. Therefore,
in order to finish the proof of Theorem 1, we need only consider bounded
rationally ergodic transformations.

\section{A gap between the limit inferior and the limit superior of symmetric
Birkhoff sums for \texttt{BRE} transformations}
\begin{thm}
\label{thm: effective divergence for symmetric sums}Let $\left(X,\BB,m,T\right)$
be an infinite, invertible, conservative, ergodic, bounded rationally ergodic, measure preserving
transformation, then 
\[
\ubeta(T)-\underbar{\ensuremath{\beta}}\left(T\right)\geq\frac{1}{5000}.
\]
\end{thm}
\begin{rem}
The constant $\delta:=\frac{1}{5000}$ was chosen so that 
\begin{equation}
\left(1-50\delta\right)\leq0.99\label{eq: 1-48epsilon is less than 0.99}
\end{equation}
and 
\begin{equation}
\frac{1}{2}\left(\frac{100}{99}\right)^{2}\frac{1+\delta}{1-\delta}\leq\frac{1}{\sqrt{3}}.\label{eq: (1000)/99^2 ... lestt than 1/sqrt(3)}
\end{equation}
We would like to point out that by a more careful bookkeeping one
can obtain a better constant for $\delta$. This will amount in more
technical arguments which we chose not to follow. As for now, we don't
know of any examples with $\overline{\beta}-\underline{\beta}<\frac{1}{2}$,
it is interesting to find out what is the minimal $\delta$ so that
there exists a conservative, ergodic infinite measure preserving transformation
$T$ with $\delta=\overline{\beta}(T)-\underline{\beta}\left(T\right)$. 
\end{rem}
\textsl{Proof:} Suppose otherwise that 
\[
\ubeta(T)-\lbeta(T)<\delta:=\frac{1}{5000}.
\]
and let
$a(n):=a_{n}(T)$.

Since $\lbeta(T)\leq1$ and $\ubeta(T)\geq1$, 
\[
1-\delta<\lbeta(T)\leq1\leq\ubeta(T)<1+\delta.
\]
Consequently for all $A\in\cF_{+}$, a.e. on $X$, 
\begin{equation}
\left(1-\delta\right)m(A)\lesssim\frac{1}{2a_{n}(T)}\Sigma_{n}(1_{A})\lesssim\left(1+\delta\right)m(A).
\label{eq:behavior of symmetric sums}
\end{equation}
We claim that 
\begin{equation}
2-2\delta<\alpha:=\alpha(T)<2+2\delta\label{eq:alpha is bounded above and below}
\end{equation}
Indeed, by (\ref{eq: beta rizpa is smaller than alpha/2}) , $\alpha\geq2\lbeta>2-2\delta$
and by (\ref{eq: beta gag is larger than alpha}) $\alpha\leq2\ubeta<2+2\delta$. 

The rest of the proof is a quantitative version of the ``single orbit''
argument in \cite{Aaronson-Kosloff-Weiss}, which we proceed to specify. 
\begin{itemize}
\item Fix $A\in\cF_{+}$. By Egorov there exists $B\in\cF_{+}\cap A$, $m(B)>\frac{3}{4}m\left(A\right)$
and $N_{0}\in\NN$ so that for all $n\geq N_{0}$ and $x\in B$, 
\[
\left(2-2\delta\right)m(A)\leq\sup_{N\geq n}\frac{1}{a(N)}S_{N}\left(1_{A}\right)(x)\leq\left(2+2\delta\right)m(A),
\]
and 
\begin{equation}
\left(2-2\delta\right)a(n)m(A)\leq\Sigma_{n}\left(1_{A}\right)(x)\leq\left(2+2\delta\right)a(n)m(A).
\label{eq: boundedness of Sigma_n and supS_n}
\end{equation}

\item Call a point $x\in B$ \textit{admissible }if 
\begin{align*}
\tag{A1} & \frac{S_{n}\left(1_{A}\right)(x)}{S_{n}\left(1_{B}\right)(x)}\xrightarrow[n\to\infty]{}\frac{m(A)}{m(B)};\\
\tag{A2} & \frac{1}{2a(n)}\Sigma_{n}\left(1_{B}\right)(x)=(1\pm\delta)m(B),\ \text{for all}\ n\geq N_{0}\text{;}\\
\tag{A3} & \sup_{N\geq n}\frac{1}{\alpha a\left(N\right)}S_{N}\left(1_{B}\right)(x)\xrightarrow[n\to\infty]{}m\left(B\right),
\end{align*}
and there exists $K\subset\NN$, an $x$- \textit{admissible subsequence}
in the sense that
\begin{align*}
\tag{A4} & T^{n}x\in B,\ \forall n\in K\ \text{and}\\
\tag{A5} & \frac{1}{\alpha a\left(n\right)}S_{n}\left(1_{B}\right)\left(x\right)\xrightarrow[n\to\infty,\ n\in K]{}m\left(B\right).
\end{align*}

\end{itemize}
An admissible pair is $\left(x,K\right)\in B\times2^{\NN}$ where
$x$ is an admissible point and $K$ is an $x$-admissible subsequence. 

Note that if $\left(x,K\right)$ is an admissible pair, then by $\left(\mbox{A1}\right)$
and $\left(\mbox{A5}\right)$ 
\[
\frac{1}{\alpha a\left(n\right)}S_{n}\left(1_{A}\right)\left(x\right)\xrightarrow[n\to\infty,\ n\in K]{}m\left(A\right).
\]

\begin{lem}
Almost every $x\in B$ is admissible. \end{lem}
\begin{proof}
By  (\ref{eq: boundedness of Sigma_n and supS_n}), the definition of $\alpha$
and the ratio ergodic theorem, almost every $x\in B$ satisfies $\left(\mbox{A1}\right)$,
$\left(\mbox{A2}\right)$ and $\left(\mbox{A3}\right)$. 

Also since $\alpha=\alpha(T)<\infty$, for a.e. $x\in B$, $\exists K\subset\NN$
satisfying $\left(\mbox{A5}\right)$. 

We claim that if $K:=\left\{ k_{n}:\ n\geq1\right\} $, $k_{n}\uparrow$,
then $K':=\left\{ k'_{n}:\ n\geq1\right\} $ where $k'_{n}=\max\left\{ j\leq k_{n}:\ T^{j}x\in B\right\} $
is $x$-admissible. Evidently $K'$ is infinite and satisfies $\left(\mbox{A4}\right)$.
To check $\left(\mbox{A5}\right)$ for $K'$: 
\[
\alpha a\left(k_{n}\right)m\left(B\right)\geq\alpha a\left(k'_{n}\right)m\left(B\right)\overset{\left(\mbox{A3}\right)}{\gtrsim}S_{k_{n}'}\left(1_{B}\right)(x)=S_{k_{n}}\left(1_{B}\right)(x)\overset{_{\left(\mbox{A5}\right)}}{\sim}\alpha a\left(k_{n}\right)m\left(B\right).
\]
 This shows that 
\[
a\left(k'_{n}\right)\sim a\left(k_{n}\right),\ \text{as}\ n\to\infty
\]
 and that 
\[
\frac{1}{\alpha a\left(k_{n}'\right)}S_{k'_{n}}\left(1_{B}\right)(x)\xrightarrow[n\to\infty]{}m(B).
\]

\end{proof}
The proof goes as follows. Lemmas \ref{lem: Bound from below implies bound from above},
\ref{lem: Lower bound and upper bound} and \ref{lem: bound on a(8n/9)}
deal with some consequences of the definition of admissable pairs
$(x,K)$ on the growth of the return sequence $a(n)$ along $n\in K$.
Then we fix an admissible pair $(x,K)$ and use these Lemmas to arrive
to a contradiction. 
\begin{lem}
\label{lem: Bound from below implies bound from above}Let $\rho\in(0,1/2)$.
If $x\in B,\ K\subset\NN$ and $\left\{ J_{n}:\ n\in K\right\} \subset\mathbb{N}$
satisfy 
\begin{align*}
 & \frac{1}{\alpha a(n)}S_{n}\left(1_{A}\right)(x)\xrightarrow[n\to\infty,\ n\in K]{}m(A);\\
 & n\geq J_{n}\xrightarrow[n\to\infty,\ n\in K]{}\infty;\\
 & \linfK\frac{a\left(J_{n}\right)}{a(n)}\geq\rho,
\end{align*}
then 
\[
\frac{1}{\alpha a\left(J_{n}\right)}S_{J_{n}}\left(1_{A}\right)(x)\underset{_{n\in K}}{\gtrsim}m(A)\left(1-\frac{4\delta}{\rho}\right).
\]
\end{lem}
\begin{proof}
Since $x\in B$, for $n\in K$ large 
\[
S_{n}\left(1_{A}\right)(x)+S_{n}^{-}\left(1_{A}\right)(x)=\Sigma_{n}\left(1_{A}\right)(x)\lesssim2\left(1+\delta\right)a(n)m(A).
\]
Consequently by (A1) and (A5), 
\[
S_{n}\left(1_{A}\right)(x)\sim\alpha a(n)m(A)
\]
and 
\begin{eqnarray*}
S_{n}^{-}\left(1_{A}\right)(x) & \underset{_{n\in K}}{\lesssim} & \left[2+2\delta-\alpha\right]a(n)m(A)\\
 & \leq & 4\delta a(n)m(A).
\end{eqnarray*}
This implies that 
\begin{eqnarray*}
\frac{1}{\alpha a\left(J_{n}\right)}S_{J_{n}}^{-}\left(1_{A}\right)\left(x\right) & \leq & \frac{1}{\alpha a\left(J_{n}\right)}S_{n}^{-}\left(1_{A}\right)(x)\underset{_{n\in K}}{\lesssim}\frac{4\delta a\left(n\right)m\left(A\right)}{\alpha a\left(J_{n}\right)}\\
 & \underset{_{n\in K}}{\lesssim} & \frac{4\delta}{\alpha\rho}m(A)
\end{eqnarray*}
and 
\begin{align*}
\frac{1}{\alpha a\left(J_{n}\right)}S_{J_{n}}\left(1_{A}\right)(x) & =\frac{1}{\alpha a\left(J_{n}\right)}\Sigma_{J_{n}}\left(1_{A}\right)(x)-\frac{1}{\alpha a\left(J_{n}\right)}S_{J_{n}}^{-}\left(1_{A}\right)\left(x\right)\\
 & \overset{_{\left(\mbox{A2}\right)}}{\gtrsim}\frac{\left(2-2\delta\right)}{\alpha}m(A)-\frac{1}{\alpha a\left(J_{n}\right)}S_{J_{n}}^{-}\left(1_{A}\right)\left(x\right)\\
 & \underset{_{n\in K}}{\gtrsim}\frac{1}{\alpha}\left[\left(2-2\delta\right)-\frac{4\delta}{\rho}\right]m\left(A\right)\geq\left(1-4\delta/\rho\right)m(A).
\end{align*}
Here the last inequality follows from 
\[
\frac{2-2\delta}{\alpha}\geq\frac{2-2\delta}{2+2\delta}\geq 1-2\delta>1-\frac{\delta}{\rho},
\]
and $\alpha>2-2\delta>3/2$. \end{proof}
\begin{lem}
\label{lem: Lower bound and upper bound}Let $\left(x,K\right)\in B\times2^{\NN}$
be an admissible pair then 
\[
\frac{2}{25}\leq\linfK\frac{a\left(\frac{n}{9}\right)}{a(n)}\ \&\ \lsupK\frac{a\left(\frac{n}{9}\right)}{a(n)}\leq\frac{1}{3}.
\]
\end{lem}
\begin{proof}
We show first that 
\begin{align*}
\tag{a} & \linfK\frac{a\left(\frac{n}{9}\right)}{a(n)}\geq\frac{2}{25}.
\end{align*}
Define for $n\in K$, 
\[
J_{l}:=\min\left\{ l\geq\frac{ln}{9}:\ T^{l}x\in B\right\} \wedge\frac{\left(l+1\right)n}{9};\ \left(0\leq l\leq8\right),
\]
then 
\begin{align*}
\alpha m\left(B\right)a\left(n\right) & \underset{_{n\in K}}{\overset{_{\left(\mbox{A5}\right)}}{\lesssim}}S_{n}\left(1_{B}\right)\left(x\right)=\sum_{l=0}^{8}S_{\frac{n}{9}}\left(1_{B}\right)\left(T^{\frac{ln}{9}}x\right)\\
 & =\sum_{l=0}^{8}S_{\frac{\left(l+1\right)n}{9}-J_{l}}\left(1_{B}\right)\left(T^{J_{l}}x\right)\leq\sum_{l=0}^{8}S_{\frac{n}{9}}\left(1_{B}\right)\left(T^{J_{l}}x\right)\\
 & \leq\sum_{l=0}^{8}S_{\frac{n}{9}}\left(1_{B}\right)\left(T^{J_{l}}x\right)\leq\sum_{l=0}^{8}\left\Vert S_{\frac{n}{9}}\left(1_{A}\right)\right\Vert _{L_{\infty}\left(B\right)}\\
 & \underset{}{\lesssim}9(2+2\delta)a\left(\frac{n}{9}\right)m\left(A\right).
\end{align*}
 Thus 
\[
\linfK\frac{a\left(\frac{n}{9}\right)}{a(n)}\geq\frac{\alpha m(B)}{18(1+\delta)m\left(A\right)}>\frac{2}{25}.\ \Box\left(\mbox{a}\right)
\]
 Next we show 
\begin{align*}
\tag{b} & \lsupK\frac{a\left(\frac{n}{3}\right)}{a(n)}\leq\frac{1}{\sqrt{3}}.
\end{align*}
By $\left(\mbox{a}\right)$ and monotonicity of $a(n)$, $\left\{ J_{n}=n/3:n\in K\right\} $
satisfies the conditions of Lemma \ref{lem: Bound from below implies bound from above}
with $\rho=2/25$, hence 
\[
S_{\frac{n}{3}}\left(1_{A}\right)(x)\underset{_{n\in K}}{\gtrsim}\alpha a\left(\frac{n}{3}\right)\left(1-50\delta\right)m(A).
\]
 By $\left(\mbox{A1}\right)$, 
\begin{equation}
S_{\frac{n}{3}}\left(1_{B}\right)(x)\underset{_{n\in K}}{\gtrsim}\alpha a\left(\frac{n}{3}\right)\left(1-50\delta\right)m(B)\overset{_{\eqref{eq: 1-48epsilon is less than 0.99}}}{\geq}\frac{99\alpha}{100}a\left(\frac{n}{3}\right)m\left(B\right).\label{eq: n/3 is a large sequence}
\end{equation}
 For $n\in K$, let 
\[
j_{n}:=\max\left\{ j\leq n/3:\ T^{j}x\in B\right\} .
\]
We claim that $a\left(j_{n}\right)\underset{n\in K}{\gtrsim}0.99a\left(n/3\right)$,
since
\begin{eqnarray*}
\alpha a\left(j_{n}\right)m(B) & \gtrsim & S_{j_{n}}\left(1_{B}\right)(x)=S_{\frac{n}{3}}\left(1_{B}\right)(x)\\
 & \underset{_{n\in K}}{\gtrsim} & \frac{99\alpha}{100}a\left(\frac{n}{3}\right)m(B).
\end{eqnarray*}
Finally since $T^{j_{n}}x\in B$, 
\begin{eqnarray*}
\left(2+2\delta\right)a\left(n\right)m(A) & \gtrsim & \Sigma_{n}\left(1_{A}\right)\left(T^{j_{n}}x\right)=\sum_{k=-n+j_{n}}^{n+j_{n}}1_{A}\left(T^{k}x\right)\\
 & \geq & \Sigma_{j_{n}}\left(1_{A}\right)\left(T^{j_{n}}x\right)+\Sigma_{j_{n}}\left(1_{A}\right)\left(T^{n}x\right)\\
 & \overset{\left(\star\right)}{\gtrsim} & 2\left(2-2\delta\right)a\left(j_{n}\right)m(A)\underset{_{n\in K}}{\gtrsim}\left(4-4\delta\right)\left(0.99a\left(\frac{n}{3}\right)\right)m(A).
\end{eqnarray*}
In $\left(\star\right)$ we used the fact that $T^{j_{n}}x,T^{n}x\in B$.
Therefore 
\[
\lsupK\frac{a\left(n/3\right)}{a(n)}\leq\frac{100}{198}\cdot\frac{\left(1+\delta\right)}{\left(1-\delta\right)}\overset{_{\eqref{eq: (1000)/99^2 ... lestt than 1/sqrt(3)}}}{\leq}\frac{1}{\sqrt{3}}.\ \Box\left(\mbox{b}\right)
\]
 Next, we show that 
\begin{align*}
\tag{c} & \lsupK\frac{a\left(\frac{n}{9}\right)}{a\left(n\right)}\leq\frac{1}{3}.
\end{align*}
For $n\in K$, let 
\[
L_{n}:=\min\left\{ J\geq\frac{n}{3}:\ T^{J}x\in B\right\} .
\]
Since $n\in K$, $T^{n}x\in B$, whence $L_{n}\leq n$. 

It follows from 
\[
\frac{a\left(L_{n}\right)}{a(n)}\geq\frac{a(n/9)}{a(n)}\underset{_{n\in K}}{\gtrsim}\frac{2}{25},
\]
and Lemma \ref{lem: Bound from below implies bound from above} that
(here we move from $1_{A}$ to $1_{B}$ using condition (A1)), 
\begin{eqnarray*}
\alpha\left(1-50\delta\right)m\left(B\right)a\left(L_{n}\right) & \underset{_{n\in K}}{\lesssim} & S_{L_{n}}\left(1_{B}\right)\left(x\right)\leq S_{n/3}\left(1_{B}\right)\left(x\right)+1\\
 & \lesssim & \alpha a\left(n/3\right)m\left(B\right),
\end{eqnarray*}
hence 
\begin{equation}
a\left(n/3\right)\underset{_{n\in K}}{\gtrsim}0.99a\left(L_{n}\right).\label{eq: L_n versus n/3}
\end{equation}

Define for $n\in K$,
\[
l_{n}:=\max\left\{ j\leq\frac{L_{n}}{3}:\ T^{j}x\in B\right\} .
\]
By repeating the previous argument with $L_{n}$ replaced by $L_{n}/3$
(which is still greater or equal to $n/9)$, by Lemma \ref{lem: Bound from below implies bound from above},
\[
S_{L_{n}/3}\left(1_{A}\right)(x)\underset{_{n\in K}}{\gtrsim}\frac{99\alpha}{100}a\left(L_{n}/3\right)\cdot m(A)
\]
and

\begin{eqnarray*}
\alpha a\left(l_{n}\right)m\left(B\right) & \gtrsim & S_{l_{n}}\left(1_{B}\right)(x)=S_{L_{n}/3}\left(1_{B}\right)(x)\\
 & \underset{_{n\in K}}{\gtrsim} & \alpha\cdot\left(0.99\right)a\left(L_{n}/3\right)m\left(B\right).
\end{eqnarray*}
Therefore 
\begin{equation}
a\left(l_{n}\right)\underset{_{n\in K}}{\gtrsim}0.99a\left(L_{n}/3\right).\label{eq:j_n vs. L_n/3}
\end{equation}
The argument in the proof of $\left(\mbox{b}\right)$ shows that 
\begin{eqnarray*}
2\left(2-2\delta\right)a\left(l_{n}\right)m\left(A\right) & \lesssim & \Sigma_{l_{n}}\left(1_{A}\right)\left(T^{l_{n}}x\right)+\Sigma_{l_{n}}\left(1_{A}\right)\left(T^{L_{n}}x\right)\\
 & \leq & \Sigma_{L_{n}}\left(1_{A}\right)\left(T^{l_{n}}x\right)\lesssim\left(2+2\delta\right)a\left(L_{n}\right)m(A).
\end{eqnarray*}
Here we used in the first inequality the fact that $T^{l_{n}}x,T^{L_{n}}x\in B$
and in the last inequality the fact $T^{l_{n}}x\in B$. 

Therefore 
\begin{equation}
a\left(l_{n}\right)\underset{_{n\in K}}{\lesssim}\left(\frac{1}{2}\right)\left(\frac{1+\delta}{1-\delta}\right)a\left(L_{n}\right),\label{eq: j_n vs. L_n}
\end{equation}
and 
\begin{eqnarray*}
\frac{a\left(n/9\right)}{a\left(n/3\right)} & \leq & \frac{a\left(L_{n}/3\right)}{a\left(n/3\right)}\overset{\eqref{eq: L_n versus n/3}}{\leq}\frac{100}{99}\cdot\frac{a\left(L_{n}/3\right)}{a\left(L_{n}\right)}\\
 & \overset{\eqref{eq:j_n vs. L_n/3}}{\underset{_{n\in K}}{\lesssim}} & \left(\frac{100}{99}\right)^{2}\frac{a\left(l_{n}\right)}{a\left(L_{n}\right)}\overset{\eqref{eq: j_n vs. L_n}}{\underset{_{n\in K}}{\lesssim}}\left(\frac{100}{99}\right)^{2}\cdot\frac{1}{2}\left(\frac{1+\delta}{1-\delta}\right)\\
 & \overset{_{\eqref{eq: (1000)/99^2 ... lestt than 1/sqrt(3)}}}{\leq} & \frac{1}{\sqrt{3}}.
\end{eqnarray*}
Finally 
\[
\frac{a\left(n/9\right)}{a\left(n\right)}=\frac{a\left(n/9\right)}{a\left(n/3\right)}\cdot\frac{a\left(n/3\right)}{a\left(n\right)}\underset{_{n\in K}}{\lesssim}\frac{1}{3}.\ \Box\left(\mbox{c}\right)
\]
\end{proof}
\begin{lem}
\label{lem: bound on a(8n/9)} If $\left(x,K\right)$ is an admissible
pair then 
\[
\lsupK\frac{a\left(\frac{8n}{9}\right)}{a\left(n\right)}\leq0.94.
\]
\end{lem}
\begin{proof}
First we show that 
\begin{equation}
S_{\frac{n}{9}}^{-}\left(1_{A}\right)\left(T^{n}x\right)\underset{_{K\ni n\to\infty}}{\gtrsim}\left(2-52\delta\right)a\left(\frac{n}{9}\right)m\left(A\right)\geq\frac{96\alpha}{100}a\left(\frac{n}{9}\right)m\left(A\right),\label{eq:S_n/9 minus is big}
\end{equation}
here the last inequality follows from $\alpha\leq2+2\delta=\frac{10002}{5000}$
and $\left(2-52\delta\right)=\frac{9648}{5000}\geq\frac{96\alpha}{100}.$

Indeed, since 
\[
S_{n}^{-}\left(1_{A}\right)\left(T^{n}x\right)=S_{n}\left(1_{A}\right)(x)\underset{_{K\ni n\to\infty}}{\sim}\alpha a(n)m\left(A\right),
\]
then 
\begin{eqnarray*}
S_{n}\left(1_{A}\right)\left(T^{n}x\right) & = & \Sigma_{n}\left(1_{A}\right)\left(T^{n}x\right)-S_{n}^{-}\left(1_{A}\right)\left(T^{n}x\right)\\
 & \underset{_{K\ni n\to\infty}}{\sim} & \Sigma_{n}\left(1_{A}\right)\left(T^{n}x\right)-\alpha a(n)m\left(A\right).
\end{eqnarray*}
In addition for every $n\in K$, $T^{n}x\in B$, it follows from $\left(\mbox{A2}\right)$
that as $K\ni n\to\infty$, 
\[
\Sigma_{n}\left(1_{A}\right)\left(T^{n}x\right)\underset{}{\lesssim}\left(2+2\delta\right)a\left(n\right)m\left(A\right).
\]
Therefore since $\alpha>2-2\delta$, 
\begin{equation}
S_{n}\left(1_{A}\right)\left(T^{n}x\right)\underset{_{K\ni n\to\infty}}{\lesssim}\left(\left(2+2\delta\right)-\alpha\right)a(n)m(A)\leq4\delta a(n)m\left(A\right).\label{eq:S_n T^n is small when n is in K}
\end{equation}
Finally 
\begin{eqnarray*}
S_{\frac{n}{9}}^{-}\left(1_{A}\right)\left(T^{n}x\right) & \geq & \Sigma_{\frac{n}{9}}\left(1_{A}\right)\left(T^{n}x\right)-S_{n}^{-}\left(1_{A}\right)\left(T^{n}x\right)\\
 & \overset{_{"\left({\rm A2}\right)\ \text{and\ \eqref{eq:S_n T^n is small when n is in K}"}}}{\underset{_{n\in K}}{\gtrsim}} & \left[\left(2-2\delta\right)a\left(\frac{n}{9}\right)-4\delta a\left(n\right)\right]m(A)\\
 & \overset{_{_{\mbox{{\rm Lemma\ \ref{lem: Lower bound and upper bound}}}}}}{\underset{_{n\in K}}{\gtrsim}} & \left[\left(2-2\delta\right)a\left(\frac{n}{9}\right)-50\delta a\left(\frac{n}{9}\right)\right]m(A)\\
 &=& \left(2-52\delta\right)m\left(A\right)a\left(\frac{n}{9}\right).\ \Box\eqref{eq:S_n/9 minus is big}
\end{eqnarray*}
Now since 
\[
\frac{a\left(\frac{8n}{9}\right)}{a(n)}\geq\frac{a\left(\frac{n}{9}\right)}{a\left(n\right)}\underset{_{n\in K}}{\gtrsim}\frac{2}{25},
\]
then by Lemma \ref{lem: Bound from below implies bound from above},
\begin{eqnarray*}
\alpha\left(1-50\delta\right)a\left(\frac{8n}{9}\right)m(A) & \underset{_{n\in K}}{\lesssim} & S_{\frac{8n}{9}}\left(1_{A}\right)\left(x\right)=S_{n}\left(1_{A}\right)\left(x\right)-S_{\frac{n}{9}}^{-}\left(1_{A}\right)\left(T^{n}x\right)\\
 & \underset{_{n\in K}}{\overset{_{\eqref{eq:S_n/9 minus is big}\ {\rm and}\ \left({\rm A5}\right)}}{\lesssim}} & \alpha m\left(A\right)\left[a\left(n\right)-\frac{96}{100}a\left(\frac{n}{9}\right)\right]\\
 & \overset{_{{\rm Lemma}\ \ref{lem: Lower bound and upper bound}}}{\underset{_{n\in K}}{\lesssim}} & \alpha m\left(A\right)a\left(n\right)\left[1-\frac{96}{100}\cdot\frac{2}{25}\right]=\alpha m\left(A\right)a\left(n\right)\left[\frac{93}{100}\right].
\end{eqnarray*}
Whence 
\[
\frac{a\left(\frac{8n}{9}\right)}{a(n)}\underset{_{n\in K}}{\lesssim}\frac{93}{100\left(1-50\delta\right)}\leq0.94.
\]

\end{proof}
\begin{proof}[Proof of Theorem \ref{thm: effective divergence for symmetric sums}]

Fix an admissible pair $\left(x,K\right)\in B\times2^{\NN}$, then
by Lemmas \ref{lem: Lower bound and upper bound} and \ref{lem: bound on a(8n/9)},
\[
\lsupK\frac{a\left(\frac{n}{9}\right)}{a\left(n\right)}\leq\frac{1}{3}\ \ \text{and}\ \ \lsupK\frac{a\left(\frac{8n}{9}\right)}{a\left(n\right)}\leq0.94.
\]
For $n\in K$, let 
\[
\mathfrak{J}{}_{n}=\mathfrak{J}_{n}\left(x\right):=\min\left\{ j\geq\frac{n}{9}:\ T^{j}x\in B\right\} .
\]
We claim that $\mathfrak{J}{}_{n}\leq\frac{8n}{9};$ else $\frac{8n}{9}<\mathfrak{J}{}_{n}\leq n$
(since for $n\in K$, $T^{n}x\in B$) and therefore as $n\to\infty,\ n\in K$:
\begin{eqnarray*}
\alpha a\left(n\right)m\left(B\right) & \sim & S_{n}\left(1_{B}\right)\left(x\right)\\
 & = & S_{\mathfrak{J}{}_{n}}\left(1_{B}\right)\left(x\right)+S_{n-\mathfrak{J}{}_{n}\vee0}\left(1_{B}\right)\left(T^{\mathfrak{J}{}_{n}}x\right)\\
 & \leq & S_{\frac{n}{9}}\left(1_{B}\right)\left(x\right)+1+S_{n-\mathfrak{J}{}_{n}\vee0}\left(1_{B}\right)\left(T^{\mathfrak{J}{}_{n}}x\right)\\
 & \overset{_{\left(\diamondsuit\right)}}{\leq} & S_{\frac{n}{9}}\left(1_{B}\right)\left(x\right)+1+S_{\frac{n}{9}}\left(1_{B}\right)\left(T^{\mathfrak{J}{}_{n}}x\right)\\
 & \overset{_{\left(\mbox{A3}\right)}}{\lesssim} & 2\alpha a\left(\frac{n}{9}\right)m(B).
\end{eqnarray*}
The inequality of $\left(\diamondsuit\right)$ is where we assume
in the contranegative that $\mathfrak{J}{}_{n}\geq\frac{8n}{9}$. 

Thus
\[
\frac{1}{2}\underset{_{n\in K}}{\lesssim}\frac{a\left(\frac{n}{9}\right)}{a\left(n\right)}\underset{_{n\in K}}{\lesssim}\frac{1}{3}.
\]
This contradiction shows that $\mathfrak{J}{}_{n}\leq\frac{8n}{9}$. 

Finally since for large $n\in K$, $\frac{n}{9}\leq\mathfrak{J}{}_{n}\leq\frac{8n}{9}$:
\[
[0,n]\subset\left[\mathfrak{J}{}_{n}-\frac{8n}{9},\mathfrak{J}{}_{n}+\frac{8n}{9}\right].
\]
Therefore as $n\to\infty,\ n\in K$, 
\begin{eqnarray*}
\left(2+2\delta\right)a\left(\frac{8n}{9}\right)m\left(A\right) & \overset{_{T^{\mathfrak{J}{}_{n}}\left(x\right)\in B}}{\gtrsim} & \Sigma_{\frac{8n}{9}}\left(1_{A}\right)\left(T^{\mathfrak{J}{}_{n}}x\right)\geq S_{n}\left(1_{A}\right)\left(x\right)\\
 & \underset{_{n\in K}}{\sim} & \alpha a\left(n\right)m\left(A\right)\geq\left(2-2\delta\right)a\left(n\right)m\left(A\right),
\end{eqnarray*}
whence by Lemma (\ref{lem: bound on a(8n/9)}), 
\[
\frac{1-\delta}{1+\delta}\underset{_{n\in K}}{\lesssim}\frac{a\left(\frac{8n}{9}\right)}{a\left(n\right)}\underset{_{n\in K}}{\lesssim} 0.94.
\]
This is a contradiction since $\frac{1-\delta}{1+\delta}=\frac{4999}{5001}>0.94$.
This proves the theorem. \end{proof}

\section{The main step to move from a return sequence to a universal bound}
\begin{lem}
\label{lem: moving from the return sequence to general sequences}Let
$\left(X,\BB,m,T\right)$ be an infinite, invertible, conservative,
bounded rationally ergodic, measure preserving transformation then
for any sequence $a_{n}\to\infty$ and for all $f\in L_{1}(X,m)_{+}$
if $0<\linf\frac{\Sigma_{n}(f)}{a_{n}}<\infty$, then 
\[
\frac{\linf\frac{\Sigma_{n}(f)}{a_{n}}}{\lsup\frac{\Sigma_{n}(f)}{a_{n}}}\leq\sqrt{\frac{\lbeta\left(T\right)}{\ubeta\left(T\right)}}.
\]
\end{lem}
\begin{proof}
Let $a_{n}\to\infty$. Assume in the contranegative that for one (equivalently
for all) $0\leq f\in L^{1}\left(X,m\right)_{+}$, 
\[
\frac{\linf\frac{\Sigma_{n}(f)}{a_{n}}}{\lsup\frac{\Sigma_{n}(f)}{a_{n}}}>\sqrt{\frac{\lbeta\left(T\right)}{\ubeta\left(T\right)}}
\]
$\lsup\frac{\Sigma_{n}(f)}{a_{n}}>0$ and $\linf\frac{\Sigma_{n}(f)}{a_{n}}<\infty$. Notice that this means that for all $f\in L_{1}(X,m)$, $\lsup\frac{\Sigma_{n}(f)}{a_{n}}<\infty$. 

Since $\left(X,\BB,m,T\right)$ is bounded rationally ergodic, there
exists $A\in\bar{\mathcal{R}}\left(T\right)$ with 
\[
0<m(A)<\infty.
\]
By multiplying $a_{n}$ by constants we can assume that, 
\[
\lsup\frac{\Sigma_{n}\left(1_{A}\right)}{a_{n}}=m\left(A\right)\ \text{and}\ \linf\frac{\Sigma_{n}\left(1_{A}\right)}{a_{n}}>\sqrt{\frac{\lbeta\left(T\right)}{\ubeta\left(T\right)}}m(A).
\]
As before, it follows from Egorov's theorem that for all $\gamma<\sqrt{\frac{\lbeta\left(T\right)}{\ubeta\left(T\right)}}<1<\lambda$,
there exists $B\subset A$ of positive measure so that for all $n$
large, 
\[
\gamma m(A)\leq\frac{\Sigma_{n}\left(1_{A}\right)(x)}{a_{n}}\leq\lambda m(A)\ \text{uniformly in }x\in B,
\]
and thus for large $n$ 
\[
\gamma m(A)m\left(B\right)\leq\int_{B}\frac{\Sigma_{n}\left(1_{A}\right)(x)}{a_{n}}dm\leq\lambda m(A)m\left(B\right).
\]
Since $\bar{\mathcal{R}}\left(T\right)$ is hereditary, $B\in\bar{\mathcal{R}}(T)$.
It follows from (\ref{eq: weak rational ergodicity for symmetric sums})
that, 
\begin{eqnarray*}
\int_{B}\frac{\Sigma_{n}\left(1_{A}\right)(x)}{a_{n}}dm & = & \frac{2a_{n}\left(T\right)}{a_{n}}\int_{B}\frac{\Sigma_{n}\left(1_{A}\right)(x)}{2a_{n}(T)}dm\\
 & \sim & \frac{2a_{n}\left(T\right)}{a_{n}}m(A)m(B).
\end{eqnarray*}
This shows that 
\[
\gamma a_{n}\lesssim2a_{n}\left(T\right)\lesssim\lambda a_{n}.
\]
Consequently for all $0\leq f\in L^{1}\left(X,m\right)$, 
\begin{eqnarray*}
\frac{\linf\frac{\Sigma_{n}\left(f\right)}{a_{n}}}{\lsup\frac{\Sigma_{n}\left(f\right)}{a_{n}}} & \leq & \frac{\lambda\linf\frac{\Sigma_{n}(f)}{2a_{n}(T)}}{\gamma\lsup\frac{\Sigma_{n}(f)}{2a_{n}(T)}}\\
 & = & \frac{\lambda}{\gamma}\frac{\lbeta\left(T\right)}{\ubeta\left(T\right)}
\end{eqnarray*}
Since $\gamma$ is arbitrary close to $\sqrt{\frac{\lbeta\left(T\right)}{\ubeta\left(T\right)}}$
and $\lambda$ is arbitrarily close to $1$, 
\[
\sqrt{\frac{\lbeta\left(T\right)}{\ubeta\left(T\right)}}<\frac{\linf\frac{\Sigma_{n}(f)}{a_{n}}}{\lsup\frac{\Sigma_{n}(f)}{a_{n}}}\leq\sqrt{\frac{\lbeta\left(T\right)}{\ubeta\left(T\right)},}\ \forall f\in L_{1}(X,m)_{+}
\]
a contradiction. \end{proof}
\begin{rem}
In \cite{Aaronson-Kosloff-Weiss} we considered two important subclasses
of infinite measure preserving transformations. Namely the ``Rank
one transformations'' and ``transformations admitting a generalized
recurrent event'' (the latter includes the class of null recurrent
Markov shifts). In those examples when (\ref{eq: Symmetric boundedness})
happens then 
\[
\frac{\underbar{\ensuremath{\beta}}\left(T\right)}{\bar{\beta}\left(T\right)}\leq\frac{1}{2}.
\]
This together with the previous Lemma shows that for those examples
for all $a_{n}\to\infty$ and $f\in L_1(X,m)_{+}$,  if  $0<\linf\frac{\Sigma_{n}(f)}{a_{n}}<\infty$, then
\[
\frac{\linf\frac{\Sigma_{n}(f)}{a_{n}}}{\lsup\frac{\Sigma_{n}(f)}{a_{n}}}\leq\frac{1}{\sqrt{2}}.
\]

\end{rem}

\section{Proof of Theorem \ref{thm: Main Theorem}}

Let $\left(X,\BB,m,T\right)$ be a conservative, ergodic, measure
preserving transformation with $m(X)=\infty$ and $a_{n}\to\infty$ such that for all $f\in L_1(X,m)_+$
\[
0<\linf\frac{\Sigma_{n}(f)}{a_{n}}<\infty.
\]
It follows from the comment after Proposition \ref{prop: AKW}  that $T$ is bounded rationally ergodic. By Lemma \ref{lem: moving from the return sequence to general sequences},
\[
\frac{\linf\frac{\Sigma_{n}(f)}{a_{n}}}{\lsup\frac{\Sigma_{n}(f)}{a_{n}}}\leq\sqrt{\frac{\lbeta\left(T\right)}{\ubeta\left(T\right)}}
\]
and by Theorem \ref{thm: effective divergence for symmetric sums}
one has 
\[
\bar{\beta}\left(T\right)-\underbar{\ensuremath{\beta}}\left(T\right)\geq\frac{1}{5000}.
\]
The theorem follows from 
\begin{eqnarray*}
\frac{\lbeta\left(T\right)}{\ubeta\left(T\right)} & \leq & \sup\left\{ \frac{y}{x}:\ y<1<x,\ |x-y|>\frac{1}{5000}\right\} \\
 & = & \sup\left\{ \frac{y}{y+1/5000}:\ y<1\right\} \\
 & = & \frac{5000}{5001}
\end{eqnarray*}
and 
\[
\sqrt{\frac{5000}{5001}}\leq1-\frac{1}{10002}.
\]

\section{Applications for horocyclic flows on geometrically finite hyperbolic
spaces}

In \cite{F. Maucourant and B. Schapira}, Maucourant and Schapira
considered the horocycle flow on geometrically finite hyperbolic spaces
and showed examples where the invariant measure is infinite yet one
still has precise knowledge of the fluctuations of the symmetric Birkhoff
integrals which we now proceed to specify. 

In this setting, let $\Gamma_{0}$ be a non elementary finitely generated
discrete subgroup of $G=SL(2,\RR)$ without Torsion elements other
than $-Id$. Equivalently the surface $S=\left.\Gamma_{0}\right\backslash \mathbb{H}$
where $\mathbb{H}$ is the hyperbolic plane, is a geometrically finite
hyperbolic surface. On the tangent bundle of $S$ one can consider
two measures. The first is the measure of maximal entropy for the
geodesic flow, also called the Bowen-Margulis or Patterson Sullivan
measure which we will denote by $m^{ps}$. This measure is supported
on $\Omega$, the non wandering set of the geodesic flow. The non
wandering set $\mathcal{E}$ of the horocyclic flow is the union of
horocycles intersecting $\Omega$. By \cite{Burger,Roblin}, the horocyclic
flow has a unique ergodic invariant probability measure of full support
on $\mathcal{E}$. This measure, denoted by $m$, is often called
the Burger-Roblin measure. The critical exponent of $\Gamma:=\pi_{1}\left(S\right)$
is defined by 
\[
\delta:=\limsup_{T\to\infty}\log\frac{1}{T}\#\left\{ \gamma\in\Gamma_{0}:\ d\left(o,\gamma o\right)\leq T\right\} ,
\]
for any fixed point $o\in\mathbb{H}$. In words $\delta$ is the exponential
growth rate of the orbits of $\Gamma$ on $\mathbb{H}$. The ergodic
theorem of \cite{F. Maucourant and B. Schapira} is the following
(We took the liberty of rephrasing it in a way that will explain the
connection with symmetric Birkhoff sums). 
\begin{thm*}
\cite{F. Maucourant and B. Schapira}(1)Let $S$ be a non elementary
geometrically finite hyperbolic surface. Let $u\in\mathcal{E}$ be
a non periodic and non wandering vector for the horocyclic flow. If
$f:T^{1}S\to\RR$ is continuous with compact support, then 
\[
\lim_{t\to\infty}\frac{1}{m_{H^{-}(u)}\left(\left(h^{s}u\right)_{|s|\leq t}\right)}\int_{-t}^{t}f\left(h^{s}u\right)ds=\frac{1}{m^{ps}\left(T^{1}S\right)}\int_{T^{1}S}fdm.
\]
Here $m_{H^{-}(u)}$ is the conditional measure of the Patterson-Sullivan
measure on the strong stable horocycle $H^{-}(u)=\left(h^{s}u\right)_{s\in\RR}$. 

(2) Writing $\tau(u):=m_{H^{-}(u)}\left(\left(h^{s}u\right)_{|s|\leq1}\right)$,
then $\tau$ is continuous and 
\\
 $m_{H^{-}(u)}\left(\left(h^{s}u\right)_{|s|\leq t}\right)=t^{\delta}\tau\left(g^{\log t}u\right)$. 

(3) If $S$ is convex cocompact, the non wandering set $\Omega\subset\mathcal{E}$
of the geodesic flow is compact, the map $\tau$ is bounded from above
and below on $\Omega$. Thus there exists constants $c_{S},C_{S}>0$
such that 
\[
\frac{c_{S}t^{\delta}}{m^{ps}\left(T^{1}S\right)}\int_{T^{1}S}fdm\lesssim\int_{-t}^{t}f\left(h^{s}u\right)ds\lesssim\frac{C_{S}t^{\delta}}{m^{ps}\left(T^{1}S\right)}\int_{T^{1}S}fdm,\ \text{as}\ t\to\infty
\]

\end{thm*}
The question arises of how close to $1$ can $\frac{c_{S}}{C_{S}}$
be? For example is it true that there exists a sequence of convex
cocompact geometrically finite surfaces $S_{n}=\Gamma_{n}\backslash\mathbb{H}$
such that 
\[
\frac{c_{S_{n}}}{C_{S_{n}}}\xrightarrow[n\to\infty]{}1?
\]
By modifying our proof for flows one sees that the answer to the last
question is negative. The proof caries on verbatim once one makes
the following adjustments:
\begin{itemize}
\item Definition of bounded rational ergodicity for flows by saying that
a measure preserving flow $\left(X,\BB,m,\left\{ \phi_{s}\right\} _{s\in\RR}\right)$
is bounded rationally ergodic if there exists a set $A\in\BB$, $0<m(A)<\infty$
with $M>0$ such that for all $x\in A$ and $T>0$,
\[
\int_{0}^{T}1_{A}\circ\phi_{s}(x)ds\leq M a_T(A)
\]
where $a_T(A):=\frac{1}{m(A)^2}\int_{0}^{T}m\left(A\cap\phi_{-s}A\right)ds$ 
\item Showing that if for a monotone increasing function $a:[0,\infty)\to\left[0,\infty\right)$
and a set $A\subset\mathcal{E}$ of positive $m$- measure,
\begin{equation}
0<c\lesssim\frac{1}{a(t)}\int_{0}^{t}1_{A}\left(h^{s}(u)\right)dt\lesssim C<\infty\label{eq: BSS for flows}
\end{equation}
for $m$ a.e. $u\in\mathcal{E}$, then the functions 
\[
\mathcal{E}\times[0,\infty)\ni\left(u,t\right)\mapsto F_{t}(u):=\frac{1}{a\left(t\right)}\int_{0}^{t}1_{A}\left(h^{s}(u)\right)dt
\]
satisfy the conditions of the Egorov type theorem for continuous parameter
flows. In fact this case is much simpler and can be verified by applying
Egorov on a discretization of the time parameter (a discrete skeleton)
and then using the equicontinuity in $t$ of the map $F_{t}.$
\item By the previous step one can carry the proof verbatim by first showing
that the flow is bounded rationally ergodic and then applying our
argument on a single orbit with minor modifications (in the definition
of the stopping times). 
\end{itemize}
The concluding statement is as follows. 
\begin{cor}
There exists a universal $\epsilon>0$ such that for any $S$ a convex
cocompact geometrically finite hyperbolic surface 
\[
\frac{c_{S}}{C_{S}}>1-\epsilon
\]
 where $c_{S},C_{S}$ are the constants defined by 
\[
c_{S}:=\frac{\liminf_{t\to\infty}\frac{1}{t^{\delta(S)}}\int_{-t}^{t}f\left(h^{s}u\right)ds}{\frac{1}{m^{ps}\left(T^{1}S\right)}\int_{T^{1}S}fdm}\ m-a.e.\ u\in\mathcal{E}
\]
and 
\[
C_{S}:=\frac{\limsup_{t\to\infty}\frac{1}{t^{\delta(S)}}\int_{-t}^{t}f\left(h^{s}u\right)ds}{\frac{1}{m^{ps}\left(T^{1}S\right)}\int_{T^{1}S}fdm}\ m-a.e.\ u\in\mathcal{E},
\]
 for any $f:T^{1}S\to\RR$ continuous with compact support. Equivalently
\begin{eqnarray*}
c_{S} & := & {\rm ess-}\liminf_{T\to\infty}\tau\left(g^{\log T}u\right)\\
C_{S} & := & {\rm ess-}\limsup_{T\to\infty}\tau\left(g^{\log T}u\right).
\end{eqnarray*}

\end{cor}
\textbf{Acknowledgements:} Theorem \ref{thm: effective divergence for symmetric sums}
is part of the Author's PhD thesis in Tel Aviv University done under
the supervision of Jon Aaronson. I would like to thank Benjamin Weiss
and Jon Aaronson for introducing me to the question of divergence
of symmetric sums and Eli Glasner for suggesting the problem. This research was supported in part by
the European Advanced Grant StochExtHomog (ERC AdG 320977).

\end{document}